\documentclass[oneside,english]{amsart}
\usepackage[T1]{fontenc}
\usepackage[latin9]{inputenc}
\usepackage{calc}
\usepackage{amstext,xcolor}
\usepackage{amsthm}
\usepackage{amssymb}
\usepackage{wasysym}

\makeatletter
\numberwithin{equation}{section}
\numberwithin{figure}{section}
\theoremstyle{plain}
\newtheorem{thm}{\protect\theoremname}
\theoremstyle{plain}
\newtheorem{lem}[thm]{\protect\lemmaname}
\theoremstyle{plain}
\newtheorem{cor}[thm]{\protect\corollaryname}
\theoremstyle{remark}
\newtheorem{rem}[thm]{\protect\remarkname}
\theoremstyle{plain}
\newtheorem{prop}[thm]{\protect\propositionname}

\numberwithin{thm}{section}
\usepackage{hyperref}

\makeatother

\usepackage{babel}
\providecommand{\corollaryname}{Corollary}
\providecommand{\lemmaname}{Lemma}
\providecommand{\propositionname}{Proposition}
\providecommand{\remarkname}{Remark}
\providecommand{\theoremname}{Theorem}

\begin{document}
\title{The atoms of operator-valued free convolutions}
\author{Serban T. Belinschi, Hari Bercovici, and Weihua Liu}
\address{CNRS-Institute de Math\'ematiques de Toulouse\\
118 Route de Narbonne\\
31062, Toulouse, France}
\email{Serban.Belinschi@math.univ-toulouse.fr}
\address{Department of Mathematics\\
Indiana University\\
Bloomington, IN 47405, USA}
\email{bercovic@indiana.edu}
\address{Department of Mathematics\\
The University of Arizona\\
617 N. Santa Rita Ave. \\
P.O. Box 210089 \\
Tucson, AZ 85721-0089 USA}
\email{weihualiu@math.arizona.edu}
\begin{abstract}
Suppose that $X_{1}$ and $X_{2}$ are two selfadjoint random variables
that are freely independent over an operator algebra $\mathcal{B}$.
We describe the possible operator atoms of the distribution of $X_{1}+X_{2}$
and, using linearization, we determine the possible eigenvalues of
an arbitrary polynomial $p(X_{1},X_{2})$ in case $\mathcal{B}=\mathbb{C}$.
\end{abstract}

\maketitle

\section{Introduction\label{sec:Introduction}}

Suppose that $\mathcal{A}$ is a von Neumann algebra, $\tau$ is a
faithful normal trace state on $\mathcal{A},$ and $X_{1},X_{2}\in\mathcal{A}$
are selfadjoint. Suppose, in addition, that $\alpha_{1},\alpha_{2}\in\mathbb{R}$
are eigenvalues of $X_{1}$ and $X_{2}$, respectively, and $p_{1},p_{2}\in\mathcal{A}$
are the orthogonal projections onto $\ker(X_{1}-\alpha_{1}1_{\mathcal{A}})$
and $\ker(X_{2}-\alpha_{2}1_{\mathcal{A}})$, respectively. If $\tau(p_{1})+\tau(p_{2})>1$,
it follows that $p=p_{1}\wedge p_{2}$ is nonzero, $\tau(p)\ge\tau(p_{1})+\tau(p_{2})-1$,
and
\[
(X_{1}+X_{2})p=X_{1}p_{1}p+X_{2}p_{2}p=\alpha_{1}p+\alpha_{2}p=(\alpha_{1}+\alpha_{2})p.
\]
Thus, $\alpha_{1}+\alpha_{2}$ is an eigenvalue of $X_{1}+X_{2}$.
It was observed in \cite{ber-voi-otaa} that the converse statement
is true if $X_{1}$ and $X_{2}$ are freely independent with respect
to $\tau$. More precisely, if $\alpha\in\mathbb{R}$ is an arbitrary
eigenvalue of $X_{1}+X_{2}$ and $p$ denotes the orthogonal projection
onto $\ker(X_{1}+X_{2}-\alpha1)$, then there exist unique $\alpha_{1},\alpha_{2}\in\mathbb{R}$
satsifying $\alpha=\alpha_{1}+\alpha_{2}$ such that (using the notation
above) $p=p_{1}\wedge p_{2}$ and $\tau(p)=\tau(p_{1})+\tau(p_{2})-1$.
We consider the analogous question in the case in which $X_{1}$ and
$X_{2}$ are freely independent over an algebra $\mathcal{B}\subset\mathcal{A}$
of `scalars' and the `eigenvalues' themselves are selfadjoint elements
of $\mathcal{B}$. Denote by $E:\mathcal{A}\to\mathcal{B}$ the trace-preserving
conditional expectation \cite[Proposition V.2.36]{take}, let $b\in\mathcal{B}$
be selfadjoint, and denote by $p$ the orthogonal projection onto
$\ker(X_{1}+X_{2}-b)$. Suppose that that $p\ne0,$ that $X_{1}$
and $X_{2}$ are freely independent with respect to $E$, and $E(p)$
is invertible. Then there exist unique selfadjoint elements $b_{1},b_{2}\in\mathcal{B}$
such that $b=b_{1}+b_{2}$ and $\ker(X_{1}-b_{1})\ne\{0\}\ne\ker(X_{2}-b_{2})$.
Moreover, if $p_{j}$ denotes the orthogonal projection onto $\ker(X_{j}-b_{j})$,
then $p=p_{1}\wedge p_{2}$ and $\tau(p)=\tau(p_{1})+\tau(p_{2})-1$. 

Similar results are true when $E(\ker(X-b))$ is only supposed to
have closed range, and this latter situation always applies if $\mathcal{B}$
is finite dimensional. This has consequences for variables that are
freely independent with respect to $\tau.$ Suppose that $X_{1}$
and $X_{2}$ are independent relative to $\tau$ and that $p$ is
a selfadjoint polynomial in two noncommutative indeterminates. Then
there exist $n\in\mathbb{N}$ and selfadjoint $n\times n$ scalar
matrices $a_{1},a_{2},b$ such that $\ker(p(x,y))\ne\{0\}$ if and
only if $\ker(a_{1}\otimes X_{1}+a_{2}\otimes X_{2}-b\otimes1_{\mathcal{A}})\ne\{0\}$.
Moreover, the variables $a_{1}\otimes X_{1}$ and $a_{2}\otimes X_{2}$
are freely independent over $M_{n}(\mathbb{C})\otimes1_{\mathcal{A}}$,
thus reducing the question about a polynomial to an equivalent one
concerning a sum \cite{shl-sk,msy,ch-shl,msw,banna-mai}. 

Our results are also proved for variables $X_{1}$ and $X_{2}$ that
are possibly unbounded but affiliated with $\mathcal{A}$. Some of
the material below is developed for $\mathcal{B}$-valued variables
in the absence of a trace. The most precise results do however require
a trace.

Earlier results in this vein were obtained in \cite{belinschi}. Of
course, these results show that atoms rarely occur for free convolutions.
Conditions under which no atoms occur at all were obtained earlier
\cite{shl-sk,msy,ch-shl,msw,banna-mai}. These works often deduce
the lack of atoms for $p(x,y)$ from strong regularity hypotheses
on $x$ and $y$ and do not always require free independence.

\section{Random variables and their distributions\label{sec:Random-variables-and-distr}}

We work in the context of $W^{*}$ operator valued probability space.
Such a space, denoted $(\mathcal{A},E,\mathcal{B})$, consists of
a von Neumann algebra $\mathcal{A}$, a von Neumann subalgebra $\mathcal{B}\subset\mathcal{A}$
that contains the unit of $\mathcal{A}$, and a faithful conditional
expectation $E:\mathcal{A}\to\mathcal{B}$ that we always assume to
be continuous relative to the $\sigma$-weak and $\sigma$-strong
topologies. When needed, $\mathcal{A}$ is supposed to act on a Hilbert
space $\mathcal{H}$ such that the $\sigma$-weak and $\sigma$-strong
topologies on $\mathcal{A}$ are induced by the weak operator and
strong operator topologies on $\mathcal{B}(\mathcal{H})$, respectively.
A \emph{random variable} in this probability space is a (possibly
unbounded) selfadjoint operator $X$ such that $(i1_{\mathcal{A}}-X)^{-1}$
belongs to $\mathcal{A}$. We denote by $\widetilde{\mathcal{A}}_{{\rm sa}}$
the collection of all such operators, and we denote by $\widetilde{A}$
the collection of formal sums of the form $X+iY$, where $X,Y\in\widetilde{\mathcal{A}}_{{\rm sa}}$. 

Given a random variable $X\in\widetilde{\mathcal{A}}_{{\rm sa}}$,
we denote by $\mathcal{B}\langle X\rangle$ the smallest von Neumann
subalgebra of $\mathcal{A}$ that contains $\mathcal{B}$ and $(i1_{\mathcal{A}}-X)^{-1}$.
Two random variables $X_{1},X_{2}\in\widetilde{\mathcal{A}}_{{\rm sa}}$
are said to have the same \emph{$\mathcal{B}$-distribution} if there
exists a $*$-algebra isomorphism $\Phi:\mathcal{B}\langle X_{1}\rangle\to\mathcal{B}\langle X_{2}\rangle$
such that $\Phi((i1_{\mathcal{A}}-X_{1})^{-1})=(i1_{\mathcal{A}}-X_{2})^{-1}$,
$E(\Phi(Y))=E(Y)$ for every $Y\in\mathcal{B}\langle X_{1}\rangle$,
and $\Phi(b)=b$ for every $b\in\mathcal{B}$. The $\mathcal{B}$-distribution
of a variable $X\in\widetilde{\mathcal{A}}_{{\rm sa}}$ is simply
its class relative to this equivalence relation. Naturally, it is
desirable to find more concrete objects related to $\mathcal{B}$
that determine entirely the $\mathcal{B}$-distribution of a random
variable. If $X$ commutes with $\mathcal{B}$, one may use the $\mathcal{B}$-valued
Cauchy transform defined by
\[
G_{X}(z)=E((z1_{\mathcal{A}}-X)^{-1}),\quad z\in\mathbb{C}\backslash\mathbb{R}.
\]
If, in addition, $X$ is bounded, one can use the $\mathcal{B}$-valued
moments $E(X^{n})$, $n\in\mathbb{N}$. The above options are inadequate
in general. If $X$ is bounded, the noncommutative version of $G_{X}$
does determine the  $\mathcal{B}$-distribution of $X$. We recall the definition
of this noncommutative function. We denote by $\mathbb{H}^{+}(\mathcal{A})$
the collection of those elements $a\in\mathcal{A}$ that have a positive,
invertible imaginary part; we indicate this condition by writing
\[
\Im a=\frac{a-a^{*}}{2i}>0.
\]
The algebra $M_{n}(\mathbb{C})\otimes\mathcal{A}=M_{n}(\mathcal{A})$
of $n\times n$ matrices over $\mathcal{A}$ is also a von Neumann
algebra and we write $\mathbb{H}_{n}^{+}(\mathcal{A})=\mathbb{H}^{+}(M_{n}(\mathcal{A}))$.
The noncommutative version of $\mathbb{H}^{+}(\mathcal{A})$ is simply
\[
\mathbb{H}_{\bullet}^{+}(\mathcal{A})=\bigcup_{n\in\mathbb{N}}\mathbb{H}_{n}^{+}(\mathcal{A}).
\]
 We also write $\mathbb{H}_{\bullet}^{-}(\mathcal{A})=-\mathbb{H}_{\bullet}^{+}(\mathcal{A})$
and $\mathbb{H}_{\bullet}^{+}=\mathbb{H}_{\bullet}^{+}(\mathbb{C})$.
The noncommutative $\mathcal{B}$-valued Cauchy transform of a random
variable $X\in\widetilde{\mathcal{A}}_{{\rm sa}}$ is the function
$G_{X}:\mathbb{H}_{\bullet}^{+}(\mathcal{B})\to\mathbb{H}_{\bullet}^{-}(\mathcal{B})$
defined by
\[
G_{X}(z)=E_{n}((b-1_{n}\otimes X)^{-1}),\quad z\in\mathbb{H}_{n}^{+}(\mathbb{\mathcal{B}}),
\]
where $E_{n}:M_{n}(\mathcal{A})\to M_{n}(\mathcal{B})$ is the conditional
expectation obtained by applying $E$ entrywise, and $1_{n}$ is the
unit matrix in $M_{n}(\mathbb{C})$. We also use the reciprocal Cauchy
transform $F_{X}:\mathbb{H}_{\bullet}^{+}(\mathcal{B})\to\mathbb{H}_{\bullet}^{+}(\mathcal{B})$
defined by
\[
F_{X}(b)=G_{X}(b)^{-1},\quad z\in\mathbb{H}_{\bullet}^{+}(\mathcal{B}).
\]

It was pointed out in \cite{wil2} that there are unbounded variables
with different $\mathcal{B}$-distributions that have identical noncommutative
Cauchy transforms. However, the noncommutative function $G_{X}$ does
determine entirely the atoms of $X$ and even the  $\mathcal{B}$-distributions of
the corresponding kernel projections. For our purposes, an \emph{atom}
of a random variable $X\in\mathcal{\widetilde{A}}_{{\rm sa}}$ is
defined to be an element $b\in\mathcal{B}_{{\rm sa}}$ with the property
that $\ker(b-X)\ne0.$ Here, $\ker(b-X)$ is understood as the greatest
projection $p\in\mathcal{A}$ with the property that $(b-X)p=0$.
In order to see how these atoms are determined, we discuss briefly
the concept of nontangential boundary limits for functions $f:\mathbb{H}^{+}\to\mathcal{A}$.
Suppose that $t_{0}\in\mathbb{R}$ and $a_{0}\in\mathcal{A}$. We
write
\[
\varangle\lim_{z\to t_{0}}f(z)=a_{0}
\]
if for every $\varepsilon>0$ there exists $\delta>0$ such that $\|f(z)-a_{0}\|<\varepsilon$
provided that $z=x+iy\in\mathbb{H}^{+}$ satisfies $|z-t_{0}|<\delta$
and $|x-t_{0}|/y<1/\varepsilon$. Observe that for every $z\in\mathbb{H}^{+}$,
the function $h_{z}:\mathbb{R\to\mathbb{C}}$ defined by
\begin{equation}
h_{z}(t)=\frac{z}{z-t},\quad t\in\mathbb{R},\label{eq:def of h_z}
\end{equation}
satisfies
\[
\varangle\lim_{z\to0}h_{z}(t)=\chi_{\{0\}}(t)=\begin{cases}
1, & t=0,\\
0, & t\ne0,
\end{cases}
\]
and
\[
|h_{z}(t)|=\left|\frac{x+iy}{x-t+iy}\right|\le\frac{|x|+y}{y}=1+\frac{|x|}{y},\quad z=x+iy.
\]
 In other words, $|h_{z}|$ is uniformly bounded as $z\to0$ such
that $|x|/y$ remains bounded. Applying these functions to an arbitrary
random variable $X\in\widetilde{\mathcal{A}}_{{\rm sa}}$, we obtain
the following result. (The second and third equalities use the $\sigma$-strong
continuity of $E$.)
\begin{lem}
\label{lem:ker X as a limit} For every $X\in\widetilde{\mathcal{A}}_{{\rm sa}}$
we have
\[
\varangle\lim_{z\to0}z(z1_{\mathcal{A}}-X)^{-1}=\ker(X)\text{ and \ensuremath{\varangle\lim_{z\to0}zG_{X}(z1_{\mathcal{A}})=E(\ker(X))}}
\]
in the $\sigma$-strong topology. More generally, writing $u(z)=z(z1_{\mathcal{A}}-X)^{-1}$
and $p=\ker(X),$ we have
\[
E((pb_{1})(pb_{2})\cdots(pb_{n-1})p)=\varangle\lim_{z\to0}E((u(z)b_{1})(u(z)b_{2})\cdots(u(z)b_{n-1})u(z))
\]
for every $n\in\mathbb{N}$ and every $b_{1},\dots,b_{n-1}\in\mathcal{B}$.
\end{lem}

Since the right hand side in the last equality can be written in terms
of the noncommutative function $G_{X}$, we see that all the moments
of $p$, and hence its  $\mathcal{B}$-distribution, are determined by $G_{X}$. The
above observation, applied to the variables $X-b$, $b\in\mathcal{B}_{{\rm sa}}$,
shows that the distribution of $\ker(X-b)$ is entirely determined
by $G_{X}$.

Later, we require a slight technical variation of Lemma \ref{lem:ker X as a limit}. 
\begin{lem}
\label{lem:convergence to ker(X) along omega}Suppose that $X\in\widetilde{\mathcal{A}}_{{\rm sa}}$
and that $f:\mathbb{R}_{+}\to\mathbb{H}^{+}(\mathcal{A})$ is such
that
\[
\lim_{y\downarrow0}\frac{f(y)}{iy}=1_{\mathcal{A}}
\]
in the $\sigma$-strong topology and $y\|(f(y)-X)^{-1}\|$ is bounded
for $y$ close to $0$. Then
\[
\lim_{y\downarrow0}iy(f(y)-X)^{-1}=\ker(X)
\]
 in the $\sigma$-strong topology.
\end{lem}

\begin{proof}
By Lemma \ref{lem:ker X as a limit}, it suffices to show that the
difference 
\[
iy(f(y)-X)^{-1}-iy(iy-X)^{-1}
\]
 converges $\sigma$-strongly to zero as $y\downarrow0$. This difference
can be rewritten as
\[
\left[iy(f(y)-X)^{-1}\right]\left[1_{\mathcal{A}}-\frac{f(y)}{iy}\right]\left[iy(iy-X)^{-1}\right].
\]
The lemma follows because, as $y\downarrow0$, the first factor remains
bounded, the middle factor converges $\sigma$-strongly to zero, and
the third factor converges $\sigma$-strongly to $\ker(X)$.
\end{proof}
Some information about $F_{X}(iy1_{\mathcal{A}})$ can be obtained
when $y\in\mathbb{R}_{+}$. Observe that the functions defined by
(\ref{eq:def of h_z}) satisfy
\[
\Re h_{iy}(t)=\frac{y^{2}}{y^{2}+t^{2}}\ge\chi_{\{0\}}(t),\quad t\in\mathbb{R}.
\]
We conclude that
\begin{equation}
\Re(iyG_{X}(iy1_{\mathcal{A}}))\ge E(\ker(X)),\quad y>0.\label{eq:Re iyG(iy) bounded below}
\end{equation}

\begin{lem}
\label{lem:iyF(iy) tends to inv of E(p)} Suppose that $X\in\widetilde{\mathcal{A}}_{{\rm sa}}$,
that $b\in\mathcal{B}_{{\rm sa}}$, and that $E(p)$ is invertible,
where $p=\ker(X-b)$. Then
\[
\lim_{y\downarrow0}\frac{1}{iy}F_{X}(b+iy1_{\mathcal{A}})=E(p)^{-1}
\]
in the $\sigma$-strong topology.
\end{lem}

\begin{proof}
Since $F_{X}(b+iy1_{\mathcal{A}})=F_{X-b}(iy1_{\mathcal{A}})$, it
suffices to prove the lemma for $b=0$. In this case, the hypothesis
and (\ref{eq:Re iyG(iy) bounded below}) imply the existence of $\delta>0$
such that $\Re(iyG_{X}(iy1_{\mathcal{A}}))\ge\delta1_{\mathcal{A}},$
and hence $\|F_{X}(iy1_{\mathcal{A}})\|/y\le1/\delta,$ 
for every $y>0$. Now, if a sequence $\{a_{n}\}_{n\in\mathbb{N}}\subset\mathcal{A}$
of invertible elements converge $\sigma$-strongly to an invertible
element $a$, and if $\sup_{n}\|a_{n}^{-1}\|<+\infty$, then $\{a_{n}^{-1}\}_{n\in\mathbb{N}}$
converges $\sigma$-strongly to $a^{-1}.$ This is easily seen from
the identity $a_{n}^{-1}-a^{-1}=a_{n}^{-1}(a-a_{n})a^{-1}$. Thus
the lemma follows because $\lim_{y\downarrow0}iyG_{X}(iy1_{\mathcal{A}})=E(p)$
according to Lemma \ref{lem:ker X as a limit}.
\end{proof}
There is a version of the preceding result that applies to the case
in which $E(p)$ has closed range, that is, if $0$ is an isolated
point in the spectrum of $E(p)$. Denote by $q$ the support projection
of $E(p)$, that is, $q=1_{\mathcal{A}}-\ker(E(p))$. Then $q\mathcal{A}q$
is a von Neumann algebra, $q\mathcal{B}q$ is a unital von Neumann
subalgebra of $q\mathcal{A}q$, and the map $E_{q}:a\to qE(a)q$,
$a\in q\mathcal{A}q$, is a faithful $\sigma$-strongly continuous
conditional expectation from $q\mathcal{A}q$ to $q\mathcal{B}q$.
If one of the following conditions is satisfied:
\begin{itemize}
\item [(a)]$X\in\mathcal{A}$, or
\item [(b)]$E$ preserves a faithful normal trace state on $\mathcal A$,
\end{itemize}
then we also have $qXq\in\widetilde{q\mathcal{A}q}_{{\rm sa}}.$
\begin{cor}
\label{cor:FqXq if E(p) has closed range}Let $X\in\widetilde{\mathcal{A}}_{{\rm sa}}$,
and let $b\in\mathcal{B}_{{\rm sa}}$, be such that $E(p)\ne0$ has
closed range, where $p=\ker(X-b)$. Set $q=1_{\mathcal{A}}-\ker(p)$
so $E_{q}(p)=qE(p)q$ is invertible in $q\mathcal{B}q$. Suppose that
$qXq\in\widetilde{q\mathcal{A}q}_{{\rm sa}}$ and set $p'=\ker(qXq-qbq)$.
Then $p'\ge p$, $E_{q}(p')>0$, and
\[
\lim_{y\downarrow0}\frac{1}{iy}F_{qXq}(qbq+iyq)=E_{q}(p')^{-1}
\]
in the $\sigma$-strong topology.
\end{cor}

\begin{proof}
We observe first that
\[
E((1_{\mathcal{A}}-q)p(1_{\mathcal{A}}-q))=(1_{\mathcal{A}}-q)E(p)(1_{\mathcal{A}}-q)=0,
\]
and thus $(1-q)p(1-q)=0$ because $E$ is faithful. This implies that
$(1-q)p=0,$ that is, $p\le q$. In particular, $(qXq-qbq)p=q(X-b)p=0$,
and this shows that $p'\ge p$. We also have $E_{q}(p')\ge E_{q}(p)$
and $E_{q}(p)$ is simply $E(p)$ regarded as an element of $q\mathcal{B}q$.
Thus $E_{q}(p')$ is invertible. The corollary follows now from Lemma
\ref{lem:iyF(iy) tends to inv of E(p)} applied to $qXq$ and $qbq$.
\end{proof}

\section{Freeness and subordination\label{sec:Freeness-and-subordination}}

Let $(\mathcal{A},E,\mathcal{B})$ be a von Neumann $\mathcal{B}$-valued
probablility space and let $X_{1},X_{2}\in\widetilde{\mathcal{A}}_{{\rm sa}}$
be two random variables. We recall from \cite{operator-val} that
$X_{1}$ and $X_{2}$ are said to be free with respect to $E$, or
simply $E$-free, if $E(a_{1}a_{2}\cdots a_{n})=0$ whenever $a_{j}\in\mathcal{B}\langle X_{i_{j}}\rangle$
are such that $E(a_{j})=0$ for $j=1,\dots,n$ and $i_{j}\ne i_{j+1}$
for $j=1,\dots,n-1$. The study of $E$-freeness is facilitated by
the fact that, in many important cases in which $X=X_{1}+X_{2}$ is
defined, the noncommutative Cauchy transform $G_{X}$ is analytically
subordinate to $G_{X_{1}}$ and to $G_{X_{2}}$. Results of this kind
go back to \cite{coalgebra}. We formulate first the relevant result
from \cite{be-ma-spe}. This result applies to the case in which $X_{1},X_{2}\in\mathcal{A}_{{\rm sa}}$,
that is, $X_{1}$ and $X_{2}$ are bounded, and it states the existence
of noncommutative analytic functions $\omega_{1},\omega_{2}:\mathbb{H}_{\bullet}^{+}(\mathcal{B})\to\mathbb{H}_{\bullet}^{+}(\mathcal{B})$
with the following properties:
\begin{equation}
F_{X}(z)=F_{X_{1}}(\omega_{1}(z))=F_{X_{2}}(\omega_{2}(z))=\omega_{1}(z)+\omega_{2}(z)-z\text{, \ensuremath{\quad} }z\in\mathbb{H}_{\bullet}^{+}(\mathcal{B}).\label{eq:F=00003DFcomp-omega}
\end{equation}

\begin{equation}
\Im\omega_{j}(z)\ge\Im z,\quad j=1,2,\ z\in\mathbb{H}_{\bullet}^{+}(\mathcal{B}).\label{eq:Im increases}
\end{equation}

In order to apply the subordination functions to the study of atoms,
we also require a version \cite{belinschi2,belinschi-vin} of the
Julia-Carath\'eodory theorem for noncommutative functions. We state
below the relevant parts of this result. (An interesting point is
that, while all the conclusions concern the `commutative' part $\omega|\mathbb{H}^{+}(\mathcal{B})$
of $\omega$, the proof uses the fact that $\omega$ is a noncommutative
function.)
\begin{thm}
\label{thm:Cara-Jul}Let $\omega:\mathbb{H}_{\bullet}^{+}(\mathcal{B})\to\mathbb{H}_{\bullet}^{+}(\mathcal{A})$
be an analytic noncommutative function. Suppose that there exist $b_{0},c_{0}\in\mathcal{B}$
such that $b_0=b_0^*,c_{0}>0$, and the quantity $\|\Im\omega(b_{0}+iyc_{0})\|/y$
is bounded if $y\in\mathbb{R}_{+}$ is close to zero. Then:
\begin{enumerate}
\item The limit
\[
\beta=\lim_{y\downarrow0}\frac{1}{y}\Im\omega(b_{0}+iyc)
\]
exists in the $\sigma$-strong topology for every $c\in\mathcal{B}$,
$c>0$, and it is strictly positive.
\item The limit
\[
b=\lim_{y\downarrow0}\omega(b_{0}+iyc)
\]
exists in the norm topology for every $c\in\mathcal{B}$, $c>0$,
it is independent of $c$, and it is selfadjoint.
\item We have 
\[
\lim_{y\downarrow0}\frac{1}{y}(\Re\omega(b_{0}+iyc)-b)=0
\]
in the $\sigma$-strong topology for every $c\in\mathcal{B}$, $c>0$.
\end{enumerate}
\end{thm}

With these tools in hand, we can analyze the consequences of $E(\ker(X-b))>0$.
\begin{thm}
\label{thm:main, p>0, no trace}Let $(\mathcal{A},E,\mathcal{B})$
be an operator valued von Neumann probability space, let $X_{1},X_{2}\in\mathcal{A}_{{\rm sa}}$
be two $E$-free random variables, and denote $X=X_{1}+X_{2}$. Suppose
that $b\in\mathcal{B}_{{\rm sa}}$ is such that $E(p)>0$, where $p=\ker(X-b)$.
Then there exist $b_{1},b_{2},\beta_{1},\beta_{2}\in\mathcal{B}_{{\rm sa}}$
with the following properties:
\end{thm}

\begin{enumerate}
\item [(i)]$b=b_{1}+b_{2},$
\item [(ii)]$\beta_{1},\beta_{2}>0$,
\item [(iii)]$\ker(X_{j}-b_{j})\ne0$, $j=1,2$,
\item [(iv)]$E(\ker((X_{j}-b_{j})\beta_{j}^{-1/2}))=\beta_{j}^{1/2}E(p)\beta_{j}^{1/2}$,
$j=1,2,$ \emph{and}
\item [(v)]$\beta_{1}+\beta_{2}-1_{\mathcal{A}}=E(p)^{-1}.$
\end{enumerate}
\begin{proof}
Let $\omega_{1}$ and $\omega_{2}$ be the subordination functions
described earlier. Equation (\ref{eq:F=00003DFcomp-omega}) shows
that
\[
\frac{\Im\omega_{1}(b+iy1_{\mathcal{A}})}{y}+\frac{\Im\omega_{2}(b+iy1_{\mathcal{A}})}{y}=1_{\mathcal{A}}+\frac{\Im F_{X}(b+iy1_{\mathcal{A}})}{y},\quad y\in\mathbb{R}_{+}.
\]
By Lemma \ref{lem:iyF(iy) tends to inv of E(p)}, the right hand side
remains bounded as $y\downarrow0$, and thus Theorem \ref{thm:Cara-Jul}
shows that the norm limits
\[
b_{j}=\lim_{y\downarrow0}\omega_{j}(b+iy1_{\mathcal{A}}),\quad j=1,2,
\]
and the strictly positive $\sigma$-strong limits
\[
\beta_{j}=\lim_{y\downarrow0}\frac{\Im\omega_{j}(b+iy1_{\mathcal{A}})}{y},\quad j=1,2,
\]
exist and, in addition,
\begin{equation}
\beta_{j}=\lim_{y\downarrow0}\frac{\omega_{j}(b+iy1_{\mathcal{A}})-b_{j}}{iy},\quad j=1,2.\label{eq:beta j is strong limit}
\end{equation}
Next, we use the subordination relation to see that
\begin{equation}
iyG_{X}(b+iy1_{\mathcal{A}})=iyG_{X_{j}}(\omega_{j}(b+iy1_{\mathcal{A}})),\quad j=1,2,\ y\in\mathbb{R}_{+}.\label{eq:iyG(iy)}
\end{equation}
Define
\[
f_{j}(y)=\beta_{j}^{-1/2}(\omega_{j}(b+iy1_{\mathcal{A}})-b_{j})\beta_{j}^{-1/2}\quad j=1,2,\ y\in\mathbb{R}_{+},
\]
and observe that (\ref{eq:beta j is strong limit}) implies
\[
\lim_{y\downarrow0}\frac{f_{j}(y)}{iy}=1
\]
 in the $\sigma$-strong topology. Therefore, by Lemma \ref{lem:convergence to ker(X) along omega},
\begin{align*}
\lim_{y\downarrow0}iy(f_{j}(y)-\beta_{j}^{-1/2}(X_{j}-b_{j})\beta_{j}^{-1/2})^{-1} & =\ker(\beta_{j}^{-1/2}(b_{j}-X_{j})\beta_{j}^{-1/2})\\
 & =\ker((b_{j}-X_{j})\beta_{j}^{-1/2})
\end{align*}
in the $\sigma$-strong topology. Now,
\begin{align*}
\omega_{j}(b+iy1_{\mathcal{A}})-X_{j} & =\omega_{j}(b+iy1_{\mathcal{A}})-b_{j}+b_{j}-X_{j}\\
 & =\beta_{j}^{1/2}f_{j}(y)\beta_{j}^{1/2}+b_{j}-X_{j}\\
 & =\beta_{j}^{1/2}(f_{j}(y)-\beta_{j}^{-1/2}(b_{j}-X_{j})\beta_{j}^{-1/2})\beta_{j}^{1/2},
\end{align*}
so
\[
(\omega_{j}(b+iy1_{\mathcal{A}})-X_{j})^{-1}=\beta_{j}^{-1/2}(f_{j}(y)-\beta_{j}^{-1/2}(b_{j}-X_{j})\beta_{j}^{-1/2})^{-1}\beta_{j}^{-1/2}
\]
and
\begin{equation}
\lim_{y\downarrow0}iy(\omega_{j}(b+iy1_{\mathcal{A}})-X_{j})^{-1}=\beta_{j}^{-1/2}\ker((b_{j}-X_{j})\beta_{j}^{-1/2})\beta_{j}^{-1/2}\label{eq:limit to be used later}
\end{equation}
 in the $\sigma$-strong topology. Similarly,
\begin{align*}
G_{X_{j}}(\omega_{j}(b+iy1_{\mathcal{A}})) & =E((\omega_{j}(b+iy1_{\mathcal{A}})-X_{j})^{-1})\\
 & =\beta_{j}^{-1/2}E((f_{j}(y)-\beta_{j}^{-1/2}(b_{j}-X_{j})\beta_{j}^{-1/2})^{-1})\beta_{j}^{-1/2}.
\end{align*}
Taking $\sigma$-strong limits in (\ref{eq:iyG(iy)}), we obtain
\[
E(p)=\beta_{j}^{-1/2}E(\ker(b_{j}-X_{j})\beta_{j}^{-1/2})\beta_{j}^{-1/2},
\]
that is, (iv). In particular, $\ker(b_{j}-X_{j})\beta_{j}^{-1/2}\ne0$,
and this implies (iii). We observe next that
\[
\omega_{1}(b+iy1_{\mathcal{A}})+\omega_{2}(b+iy1_{\mathcal{A}})=b+iy1_{\mathcal{A}}+F_{X}(iy).
\]
The left side tends in norm to $b_{1}+b_{2}$ as $y\downarrow0$,
while
\[
F_{X}(iy)=iy\frac{F_{X}(iy)}{iy}\to0\cdot E(p)=0
\]
as $y\downarrow0$. This proves (i). Finally, since
\[
\frac{1}{iy}[\omega_{1}(b+iy1_{\mathcal{A}})-b_{1}]+\frac{1}{iy}[\omega_{2}(b+iy1_{\mathcal{A}})-b_{2}]=1_{\mathcal{A}}+\frac{1}{iy}F_{X}(iy),
\]
we obtain by letting $y\downarrow0$
\begin{equation}
\beta_{1}+\beta_{2}=1_{\mathcal{A}}+E(p)^{-1}.\label{eq:beta1+beta2}
\end{equation}
\end{proof}
\begin{rem}
Since
\[
\beta_{j}^{1/2}E(\ker((X_{j}-b_{j})\beta_{j}^{-1/2}))\beta_{j}^{-1/2}=\beta_{j}E(p),\quad j=1,2,
\]
we obtain 
\[
\beta_{1}^{1/2}E(\ker[(X_{1}-b_{1})\beta_{1}^{-1/2}])\beta_{1}^{-1/2}+\beta_{2}^{1/2}E(\ker[(X_{2}-\beta_{2})\beta_{2}^{-1/2}]\beta_{2}^{-1/2}=1_{\mathcal{A}}+E(p)
\]
 upon multiplying (\ref{eq:beta1+beta2}) on the right by $E(p)$.
A similar equation is obtained when we multiply on the left by $E(p)$
(or when we take adjoints in the above equation.)
\end{rem}

\begin{rem}
\label{rem:vN-M equivalence of ker} Since $\ker(b_{j}-X_{j})$ and
$\ker((b_{j}-X_{j})\beta_{j}^{-1/2})$ are the left and right support
projections of $\beta_{j}^{1/2}\ker(b_{j}-X_{j})$, it follows that
these two projections are Murray-von Neumann equivalent in $\mathcal{A}$.
\end{rem}

Suppose now that $\mathcal{A}$ is a von Neumann algebra, $\mathcal{B\subset\mathcal{A}}$
is a von Neumann subalgebra containing the unit of $\mathcal{A}$,
and $\tau:\mathcal{A}\to\mathbb{C}$ is a normal faithful trace state.
We denote by $E_{\mathcal{B}}:\mathcal{A}\to\mathcal{B}$ the unique
trace preserving conditional expectation, that is, $\tau\circ E_{\mathcal{B}}=\tau$,
so $(\mathcal{A},E_{\mathcal{B}},\mathcal{B})$ is an operator valued
probability space. In this context, the formal set $\widetilde{\mathcal{A}}$ introduced at the beginning of Section \ref{sec:Random-variables-and-distr}
has an algebra structure; in particular, $\widetilde{\mathcal{A}}_{{\rm sa}}$
is a vector space. Thus the addition of arbitrary random variables
in $\widetilde{\mathcal{A}}_{{\rm sa}}$ is defined. Suppose that
$X_{1},X_{2}\in\widetilde{\mathcal{A}}_{{\rm sa}}$ and set $X=X_{1}+X_{2}$.
It was shown in \cite{coalgebra,freeMark} that noncommutative functions
$\omega_{1}$ and $\omega_{2}$ satisfying (\ref{eq:F=00003DFcomp-omega})
and (\ref{eq:Im increases}) do exist. In addition, the stronger subordination
equation
\begin{equation}
E_{\mathcal{B}\langle X_{j}\rangle}((b-X)^{-1})=(\omega_{j}(b)-X_{j})^{-1},\quad j=1,2,\ b\in\mathbb{H}^{+}(\mathcal{B}),\label{eq:conditional exp subordination}
\end{equation}
holds. Theorem \ref{thm:main, p>0, no trace} can be strengthened
as follows.
\begin{thm}
\label{thm:main, p>0, with trace}Let $\mathcal{A}$ be a von Neumann
algebra with a faithful normal trace state $\tau$, let $\mathcal{B}\subset\mathcal{A}$
be a von Neumann subalgebra containing the unit of $\mathcal{A}$,
let $X_{1},X_{2}\in\widetilde{\mathcal{A}}_{{\rm sa}}$ be two $E_{\mathcal{B}}$-free
random variables, and set $X=X_{1}+X_{2}$. Suppose that $b\in\mathcal{B}_{{\rm sa}}$
is such that $E(p)>0$, where $p=\ker(X-b)$. Then there exist $b_{1},b_{2},\beta_{1},\beta_{2}\in\mathcal{B}_{{\rm sa}}$
satisfying properties \emph{(i)-(v)} of Theorem\emph{ \ref{thm:main, p>0, no trace}}
and, in addition,
\begin{enumerate}
\item [(vi)]$\ker((b_{j}-X_{j})\beta_{j}^{-1/2})=\beta_{j}^{1/2}E_{\mathcal{B}\langle X_{j}\rangle}(p)\beta_{j}^{1/2},$
and
\item [(vii)]$p=p_{1}\wedge p_{2}$ and $\tau(p_{1})+\tau(p_{2})=1+\tau(p)$,
where $p_{j}=\ker(b_{j}-X_{j}),$ $j=1,2$.
\end{enumerate}
\end{thm}

\begin{proof}
Using the notation in the proof of Theorem \ref{thm:main, p>0, no trace},
the equation
\[
iyE_{\mathcal{B}\langle X_{j}\rangle}((b+iy1_{\mathcal{A}}-X)^{-1})=iy(\omega_{j}(b+iy1_{\mathcal{A}})-X_{j})^{-1}
\]
and (\ref{eq:limit to be used later}) yield (vi) as $y\downarrow0$.
To prove (vii), observe that
\[
\tau(\beta_{j}^{1/2}E(\ker((X_{j}-b_{j})\beta_{j}^{-1/2}))\beta_{j}^{-1/2})=\tau(E(\ker((X_{j}-b_{j})\beta_{j}^{-1/2}))).
\]
Remark \ref{rem:vN-M equivalence of ker} implies 
\[
\tau(E(\ker((X_{j}-b_{j})\beta_{j}^{-1/2})))=\tau(p_{j}).
\]
Thus the equality $\tau(p_{1})+\tau(p_{2})=1+\tau(p)$ follows by
applying $\tau$ in Remark \ref{rem:vN-M equivalence of ker}. Finally,
we certainly have $p\ge p_{1}\wedge p_{2}$ and
\[
\tau(p)=\tau(p_{1})+\tau(p_{2})-1\le\tau(p_{1}\wedge p_{2}).
\]
We conclude that $p=p_{1}\wedge p_{2}$ because $\tau$ is faithful.
This concludes the proof.
\end{proof}

\section{Matrix valued random variables\label{sec:Matrix-valued-random-vars}}

Let $\mathcal{A}$ be a von Neumann algebra endowed with a faithful,
normal trace state $\tau$ and let $n\in\mathbb{N}$. The algebra
$M_{n}(\mathcal{A})=M_{n}(\mathbb C)\otimes\mathcal{A}$ is also a von Neumann
algebra and the map $\tau_{n}:M_{n}(\mathcal{A})\to\mathbb{C}$ defined
by
\[
\tau_{n}(a)=\frac{1}{n}\sum_{j=1}^{n}\tau(a_{jj}),\quad a=[a_{ij}]_{i,j=1}^{n}\in M_{n}(\mathcal{A}),
\]
is faithful, normal trace state. The trace-preserving conditional
expectation $E_{n}:M_{n}(\mathbb C)\otimes\mathcal{A}\to M_{n}(\mathbb C)\otimes1_{\mathcal{A}}$
is given by
\[
E_{n}(a)=[\tau(a_{ij})]_{i,j=1}^{n},\quad a=[a_{ij}]_{i,j=1}^{n}\in M_{n}(\mathcal{A}),
\]
and therefore we have
\[
\tau_{n}={\rm tr}_{n}\circ E_{n},
\]
where ${\rm tr}_{n}$ denotes the normalized trace on $M_{n}(\mathbb C)\otimes1_{\mathcal{A}}.$

The relevance of this example comes from the following fact.
\begin{prop}
\label{prop:freeness over M_n}Suppose that the variables $X_{1},X_{2}\in\widetilde{\mathcal{A}}_{{\rm sa}}$
are $\tau$-free and that $a_{1},a_{2}\in M_{n}(\mathbb C)$ are selfadjoint
matrices. Then $a_{1}\otimes X_{1}$ and $a_{2}\otimes X_{2}$ are
$E_{n}$-free.
\end{prop}

Suppose now that $X\in\widetilde{\mathcal{A}}_{{\rm sa}}$ is a random
variable, $n\in\mathbb{N},$ and $a,b\in M_{n}(\mathbb C)$ are selfadjoint matrices.
We wish to describe the kernel of the random variable
\[
b\otimes1_{\mathcal{A}}-a\otimes X\in\widetilde{M_{n}(\mathcal{A})}_{{\rm sa}}.
\]
To do this, it is convenient to view the von Neumann algebra $M_{n}(\mathbb C)\langle a\otimes X\rangle$
(that is, the algebra generated by $M_{n}(\mathbb{C})\otimes1_{\mathcal{A}}$
and $a\otimes X$) as an algebra of operators on a concrete Hilbert
space. Denote by $\mu$ the distribution of $X$, that is, $\mu$
is the Borel measure on $\mathbb{R}$ defined by
\[
\mu(\sigma)=\tau(e_{X}(\sigma))
\]
for every Borel set $\sigma\subset\mathbb{R},$ where $e_{X}$ denotes
the spectral measure of $X$. Set 
\[
\mathcal{H}=\mathbb{C}^{n}\otimes L^{2}(\mu),
\]
and let $1_{n}\otimes X$ and $M_{n}(\mathbb C)\otimes1_{\mathcal{A}}$ act on
$f\in\mathcal{H}$ via the formulas
\[
((1_{n}\otimes X)f)(t)=tf(t),\quad((b\otimes1_{\mathcal{A}})f)(t)=bf(t),\quad b\in M_{n}(\mathbb C),t\in\mathbb{R}.
\]
Here, the elements of $\mathcal{H}$ are viewed as measurable functions
$f:\mathbb{R}\to\mathbb{C}^{n}$. Using this action, the algebra $M_{n}(\mathbb C)\langle a\otimes X\rangle$
is identified with the algebra consisting of all (equivalence classes
of) bounded Borel functions $h:\mathbb{R}\to M_{n}(\mathbb C)$ and
\[
(hf)(t)=h(t)f(t),\quad h\in M_{n}(\mathbb C)\langle a\otimes X\rangle,f\in\mathcal{H}.
\]
The maps $E_{n}$ and $\tau_{n}$ become simply
\[
E_{n}(h)=\int_{\mathbb{R}}h(t)\,d\mu_{X}(t)\text{ and }\tau_{n}(h)=\int_{\mathbb{R}}{\rm tr}_{n}(h(t))\,d\mu_{X}(t),\quad h\in M_{n}(\mathbb C)\langle a\otimes X\rangle.
\]

Returning to the operator $b\otimes1_{\mathcal{A}}-a\otimes X$, we
see that the equation
\[
(b\otimes1_{\mathcal{A}}-a\otimes X)f=0
\]
translates to
\[
(b-ta)f(t)=0,\quad\mu_{X}\text{-a.e.}
\]
It is well known (see, for instance, \cite[Section 2.1]{Kato}) that
\[
t\mapsto\ker(b-ta)
\]
is a Borel function from $\mathbb{R}\to M_{n}(\mathbb C)$ and that the rank
of the projection $\ker(b-ta)$ is equal to its minimum value outside
a finite set in $\mathbb{R}$. The result below follows immediately.
\begin{lem}
\label{lem:dimension of a kernel over M_n}Suppose that $\mathcal{A}$
is a von Neumann algebra with a normal, faithful trace $\tau$, let
$n\in\mathbb{N},$ let $X\in\widetilde{\mathcal{A}}_{{\rm sa}}$ be
a random variable, and let $a,b\in M_{n}(\mathbb C)$ be selfadjoint. Define
$k(t)={\rm tr}_{n}(\ker(b-ta))$, $t\in\mathbb{R}$, and let $k_{{\rm min}}=\min\{k(t):t\in\mathbb{R}\}$.
Then
\[
\tau_{n}(\ker(b\otimes1_{\mathcal{A}}-a\otimes X))=\int_{\mathbb{R}}k(t)\,d\mu_{X}(t)=k_{{\rm min}}+\sum_{t\in\mathbb{R}}(k(t)-k_{{\rm min}})\mu_{X}(\{t\}).
\]
\end{lem}

The sum in the lemma above only contains finitely many nonzero terms,
corresponding to those $t\in\mathbb{R}$ such that $k(t)>k_{{\rm min}}$
and at the same time $\mu_{X}(\{t\})>0$. Theorem \ref{thm:main, p>0, with trace}
yields the following result.
\begin{cor}
\label{cor:main with trace applied to B=00003DM_n}Suppose that $\mathcal{A}$
is a von Neumann algebra with a normal, faithful trace $\tau$, let
$n\in\mathbb{N},$ let $X_{1},X_{2}\in\widetilde{\mathcal{A}}_{{\rm sa}}$
be $\tau$-free random variables, and let $a_{1},a_{2},b\in M_{n}(\mathbb C)$
be selfadjoint. If $\xi=E_{n}(\ker(b\otimes1_{\mathcal{A}}-a_{1}\otimes X_{1}-a_{2}\otimes X_{2}))>0$,
then there exist $t_{1},\dots,t_{N},s_{1},\dots,s_{M}\in\mathbb{R}$
and $\ell_{0},\ell_{1},\dots,\ell_{N},m_{0},m_{1},\dots,m_{M}\in\mathbb{N}$
such that $\ell_{0}+\ell_{j}\le n$, $m_{0}+m_{i}\le n$, and
\[
n({\rm tr}_n(\xi)+1)=\ell_{0}+m_{0}+\sum_{j=1}^{N}\ell_{j}\mu_{X_{1}}(\{t_{j}\})+\sum_{i=1}^{M}m_{i}\mu_{X_{2}}(\{s_{i}\}).
\]
 If neither $\mu_{X_{1}}$ nor $\mu_{X_{2}}$ have point masses, then
$n{\rm tr}_n(\xi)$ is an integer.
\end{cor}

\begin{proof}
As noted above, the variables $a_{1}\otimes X_{1},a_{2}\otimes X_{2}\in\widetilde{M_{n}(\mathcal{A})}$
are $E_{n}$-free and thus the conclusions of Theorem \ref{thm:main, p>0, with trace}
apply to them. Thus, there exist selfadjoint elements $b_{1}\otimes1_{\mathcal{A}},b_{2}\otimes1_{\mathcal{A}}$
such that
\[
1+{\rm tr}_n(\xi)=\tau_{n}(p_{1})+\tau_{n}(p_{2}),
\]
where $p_{j}=\ker(b_{j}\otimes1_{\mathcal{A}}-a_{j}\otimes X_{j})$,
$j=1,2$. Setting 
\begin{align*}
%\ell
k_{j}(t) & ={\rm tr}_{n}(\ker(b_{j}-ta_{j})),\quad t\in\mathbb{R},\\
\ell_{0} & =n\min\{k%\ell
_{1}(t):t\in\mathbb{R}\},\\
m_{0} & =n\min\{k_{2}(t):t\in\mathbb{R}\},\\
\ell(t) & =nk_{1}(t)-\ell_{0},\\
m(t) & =nk_{2}(t)-m_{0},
\end{align*}
we conclude that 
\[
n(1+{\rm tr}_n(\xi))=\ell_{0}+m_{0}+\sum_{t\in\mathbb{R}}(\ell(t)\mu_{X_{1}}(\{t\})+m(t)\mu_{X_{2}}(\{t\})).
\]
If neither $\mu_{X_{1}}$ nor $\mu_{X_{2}}$ has any point masses,
the second sum vanishes, and thus $n{\rm tr}_n(\xi)=\ell_{0}+m_{0}-n$ is an integer.
The corollary follows.
\end{proof}
The ordinary eigenvalues of an arbitrary polynomial $P(X_{1},X_{2})$
in two random variables can be studied using the matrix eigenvalues
of an expression of the form $a_{1}\otimes X_{1}+a_{2}\otimes X_{2}$.
This is achieved by the process of linerization that we now describe
briefly. Suppose that $P(Z_{1},Z_{2})$ is a complex polynomial in
two noncommuting indeterminates and let $a_{0},a_{1},a_{2}\in M_{n}(\mathbb C)$
for some $n\in\mathbb{N}.$ We say that the expression
\[
L(Z_{1},Z_{2})=a_{0}\otimes1+a_{1}\otimes Z_{1}+a_{2}\otimes Z_{2}
\]
is a linearization of $P(Z_{1},Z_{2})$ if, given elements $z_{1},z_{2}$
in some complex unital algebra $\mathcal{A}$ and $\lambda\in\mathbb{C}$,
then the element $P(z_{1},z_{2})$ is invertible in $\mathcal{A}$
if and only if $L(z_{1},z_{2})$ 
is invertible in $M_{n}(\mathcal{A})$. It is known that every polynomial
has a linearization. As seen in \cite{anderson}, if $P(Z_{1},Z_{2})$
is selfadjoint (relative to the involution that fixes $Z_{1}$ and
$Z_{2}$), then $a_{0},a_{1},$ and $a_{2}$ can be chosen to be selfadjoint
matrices. One way to construct a linearization is to find $m\in\mathbb{N}$
and polynomials $B(Z_{1},Z_{2}),C(Z_{1},Z_{2}),D(Z_{1},Z_{2})$, and
$D'(Z_{1},Z_{2})$ such that
\begin{enumerate}
\item [(a)]$B$ is a $1\times m$ linear polynomial,
\item [(b)]$C$ is an $m\times1$ linear polynomial,
\item [(c)]$D$ is an $m\times m$ linear polynomial,
\item [(d)]$D(Z_{1},Z_{2})D'(Z_{1},Z_{2})=D'(Z_{1},Z_{2})D(Z_{1},Z_{2})=1_{m}$,
and
\item [(e)]$B(Z_{1},Z_{2})D'(Z_{1},Z_{2})C(Z_{1},Z_{2})=P(Z_{1},Z_{2}).$
\end{enumerate}
Once such polynomials are found, 
\begin{equation}
L(Z_{1},Z_{2})=\left[\begin{array}{cc}
0 & B(Z_{1},Z_{2})\\
C(Z_{1},Z_{2}) & D(Z_{1},Z_{2})
\end{array}\right]\label{eq:linearization of(Z,Z)}
\end{equation}
is a linearization of $P$ with $n=1+m$. This linearization is selfadjoint
if $C^{*}=B$ and $D^{*}=D$. 
\begin{lem}
\label{lem:kerP(X_1,X_2) equiv to kernel of linearization}Let $\mathcal{A}$
be a von Neumann algebra with a faithful normal trace state $\tau$,
and let $A,X_{1},X_{2}\in\widetilde{\mathcal{A}}_{{\rm sa}}$ be random
variables. Suppose that $P(Z_{1},Z_{2})$ is a polynomial in two noncommuting
indeterminates and that $L(Z_{1},Z_{2})$ is a linearization of $P$
defined by \emph{(\ref{eq:linearization of(Z,Z)})}, where $B,C,D,D'$
are subject to conditions \emph{(a)\textendash (e)}. Denote by $e_{1,1}\in M_{n}(\mathbb C)$
the matrix unit whose only nonzero is in the first row and first column.
Then $\ker(A\otimes e_{1,1}+L(X_{1},X_{2}))$ is Murray-von Neumann
equivalent to $(A-\ker P(X_{1},X_{2}))\oplus(0_{n-1}\otimes1_{\mathcal{A}})$
in $M_{n}(\mathcal{A}).$ In particular,
\[
n\tau_{n}(\ker(A\otimes e_{1,1}+L(X_{1},X_{2})))=\tau(\ker(A-P(X_{1},X_{2}))).
\]
\end{lem}

\begin{proof}
We first observe that
\[
A\otimes e_{1,1}+L(X_{1},X_{2})=\left[\begin{array}{cc}
A & B(Z_{1},Z_{2})\\
C(Z_{1},Z_{2}) & D(Z_{1},Z_{2})
\end{array}\right],
\]
and
\[
\left[\begin{array}{cc}
\!\!1_{\mathcal{A}} &\!\!-B(X_{1},X_{2})D'\!(X_{1},X_{2})\!\!\\
0 & 1_{n-1}\otimes1_{\mathcal{A}}
\end{array}\right]\!\!(A\otimes e_{1,1}+L(X_{1},X_{2}))\!=\!\left[\begin{array}{cc}
\!\!A\!-\!P(X_{1},X_{2}) & 0\\
C(X_{1},X_{2}) & D(X_{1},X_{2})\!\!
\end{array}\right].
\]
Since the first operator on the left side is injective, we deduce
that
\[
\ker(A\otimes e_{1,1}+L(X_{1},X_{2}))=\ker\left[\begin{array}{cc}
A-P(X_{1},X_{2}) & 0\\
C(X_{1},X_{2}) & D(X_{1},X_{2})
\end{array}\right].
\]
Next, we note the identities
\[
\left[\begin{array}{cc}
A-P(X_{1},X_{2}) & 0\\
0 & 1_{n-1}\otimes1_{\mathcal{A}}
\end{array}\right]H=\left[\begin{array}{cc}
A-P(X_{1},X_{2}) & 0\\
C(X_{1},X_{2}) & D(X_{1},X_{2})
\end{array}\right]
\]
 and
\[
\left[\begin{array}{cc}
A-P(X_{1},X_{2}) & 0\\
C(X_{1},X_{2}) & D(X_{1},X_{2})
\end{array}\right]K=\left[\begin{array}{cc}
A-P(X_{1},X_{2}) & 0\\
0 & 1_{n-1}\otimes1_{\mathcal{A}}
\end{array}\right],
\]
where
\[
H=\left[\begin{array}{cc}
1_{\mathcal{A}} & 0\\
C(X_{1},X_{2}) & D(X_{1},X_{2})
\end{array}\right]
\]
and
\[
K=\left[\begin{array}{cc}
1_{\mathcal{A}} & 0\\
-D'(X_{1},X_{2})C(X_{1},X_{2}) & D'(X_{1},X_{2})
\end{array}\right]
\]
are injective operators. The first identity shows that the final space
of 
\[
H\ker\left[\begin{array}{cc}
A-P(X_{1},X_{2}) & 0\\
C(X_{1},X_{2}) & D(X_{1},X_{2})
\end{array}\right]=H\ker(A\otimes e_{1,1}+L(X_{1},X_{2}))
\]
is less than or equal to
\[
\ker\left[\begin{array}{cc}
-P(X_{1},X_{2}) & 0\\
0 & 1_{n-1}\otimes1_{\mathcal{A}}
\end{array}\right]=(\ker(A-P(X_{1},X_{2})))\oplus(0_{n-1}\otimes1_{\mathcal{A}}),
\]
and thus
\[
\ker(A\otimes e_{1,1}+L(X_{1},X_{2}))\prec(\ker(A-P(X_{1},X_{2})))\oplus(0_{n-1}\otimes1_{\mathcal{A}}).
\]
Similarly, the second equality yields
\[
(\ker(A-P(X_{1},X_{2})))\oplus(0_{n-1}\otimes1_{\mathcal{A}})\prec\ker(A\otimes e_{1,1}+L(X_{1},X_{2}))
\]
and concludes the proof.
\end{proof}
The preceding proof yields a description of $\ker(A\otimes e_{1,1}+L(X_{1},X_{2}))$
that we note for further use.
\begin{cor}
\label{cor:ker of linearization}With the notation of Lemma \emph{\ref{lem:kerP(X_1,X_2) equiv to kernel of linearization}},
$\ker(A\otimes e_{1,1}+L(X_{1},X_{2}))$ is the final support of the
operator
\[
\left[\begin{array}{c}
q\\
-D'(X_{1},X_{2})C(X_{1},X_{2})q
\end{array}\right],
\]
where $q=\ker(A-P(X_{1},X_{2}))$.
\begin{cor}
\label{cor:when E(p) is large}With the notation of Lemma \emph{\ref{lem:kerP(X_1,X_2) equiv to kernel of linearization}},
suppose that $A\in\mathbb{C}1_{\mathcal{A}}$, $X_{1}$ and $X_{2}$
are free, $0\neq\ker(A-P(X_{1},X_{2}))\neq1_{\mathcal{A}}$, and set $p=\ker(e_{1,1}\otimes A+L(X_{1},X_{2}))$.
If $E_{n}(p)>0$ then either $X_{1}$ or $X_{2}$ has an eigenvalue.
\end{cor}

\end{cor}

\begin{proof}
By Lemma \ref{lem:kerP(X_1,X_2) equiv to kernel of linearization},
$n\tau_{n}(p)=\tau(A-P(X_{1},X_{2}))\in(0,1)$. In particular, $n\tau_{n}(p)$
is not an integer. The corollary follows from Corollary \ref{cor:main with trace applied to B=00003DM_n}.
\end{proof}
The following result allows us to treat cases in which $E_{n}(p)$
is not invertible.
\begin{lem}
\label{lem:bigger projections}Let $\mathcal{A}$ be a von Neumann
algebra with a faithful normal trace state $\tau$, and let $X\in\widetilde{M_{n}(\mathcal{A})}_{{\rm sa}}$.
Then there exist projections $q_{1},q_{2}\in M_{n}(\mathbb C)\otimes1_{\mathcal{A}}$
such that
\begin{enumerate}
\item $E_{n}(\ker(q_{1}Xq_{2}))$ and $E_{n}(\ker(q_{2}Xq_{1}))$ are invertible,
and
\item $n\tau_{n}(\ker(q_{1}Xq_{2}))-n\tau_{n}(\ker X)=n\tau_{n}(\ker(q_{2}Xq_{1}))-n\tau_{n}(\ker X)$
are integers. 
\end{enumerate}
\end{lem}

\begin{proof}
Set $p=\ker(X)$ and define $q_{1}$ to be the support projection
of $E_{n}(p)$, that is,
\[
q_{1}=1_{M_{n}(\mathcal{A})}-\ker(E_{n}(p)).
\]
 Since 
\[
E_{n}((1_{M_{n}(\mathcal{A})}-q_{1})p(1_{M_{n}(\mathcal{A})}-q_{1}))=(1_{M_{n}(\mathcal{A})}-q_{1})E_{n}(p)(1_{M_{n}(\mathcal{A})}-q_{1})=0,
\]
 and $E_{n}$ is faithful, we conclude that
\[
(1_{M_{n}(\mathcal{A})}-q_{1})p(1_{M_{n}(\mathcal{A})}-q_{1})=0.
\]
 Thus, $(1_{M_{n}(\mathcal{A})}-q_{1})p=0$, or equivalently, $p\le q_{1}$.
We show next that
\[
\ker(Xq_{1})=p+(1_{M_{n}(\mathcal{A})}-q_{1}).
\]
In fact, it is clear that $Xq_{1}p=Xp=0$ and $Xq_{1}(1-q_{1})=0$.
Moreover, $\ker(Xq_{1})$ cannot contain any nonzero projection $r$
orthogonal to $p+(1_{M_{n}(\mathcal{A})}-q_{1})$; such a projection
satisfies $r\le q_{1}$ so $Xq_{1}r=Xr\ne0$ because $r\le1-\ker(X)$.
We have
\[
E_{n}(\ker(Xq_{1}))=E_{n}(p)+(1_{M_{n}(\mathcal{A})}-q_{1}),
\]
so $E_{n}(\ker(Xq_{1}))$ is invertible and, in addition, 
\[
n\tau_{n}(\ker(Xq_{1}))-n\tau_{n}(p)=n\tau_{n}(\ker(Xq_{1}))-n\tau_{n}(E_{n}(p))=n\tau_{n}(1_{M_{n}(\mathcal{A})}-q_{1})
\]
is an integer, namely the rank of $1_{M_{n}(\mathcal{A})}-q_{1}$
viewed as a projection in $M_{n}(\mathbb C)$.

We observe next that, since $q_{1}X=(Xq_{1})^{*}$, $\ker(q_{1}X)$
is Murray-von Numann equivalent to $\ker(Xq_{1})$, and in particular
they have the same trace. We apply the preceding observation with
$q_{1}X$ in place of $X$. That is, we define $q_{2}$ to be the
support of $E_{n}(\ker(q_{1}X))$. Then the above arguments show that
$E_{n}(\ker(q_{1}Xq_{2}))$ is invertible and $n\tau_{n}(\ker(q_{1}Xq_{2}))-n\tau_{n}(\ker(q_{1}X))$
is an integer. We conclude that (2) is true because $\ker(q_{2}Xq_{1})$
is Murray-von Neumann equivalent to $\ker(q_{1}Xq_{2})$. Finally,
observe that $E_{n}(\ker(q_{2}Xq_{1}))\ge E_{n}(\ker(Xq_{1}))$ must
also be invertible, thus concluding the proof of (1).
\end{proof}
\begin{prop}
\label{prop:E(p) not invertible}Let $\mathcal{A}$ be a von Neumann
algebra with a faithful normal trace state $\tau$, let $X_{1},X_{2}\in\widetilde{\mathcal{A}}_{{\rm sa}}$
be two free random variables, and let $P(Z_{1},Z_{2})$ be a polynomial
in two noncommuting indeterminates such that $\tau(\ker(P(X_{1},X_{2})))\notin\{0,\frac{1}{2},1\}$.
Then either $X_{1}$ or $X_{2}$ has an eigenvalue.
\end{prop}

\begin{proof}
Replacing $P$ by $P^{*}P$ does not change the kernel, so we suppose
that $P$ is selfadjoint. Let 
\[
L(Z_{1},Z_{2})=L(Z_{1},Z_{2})^*=\left[\begin{array}{cc}
0 & B(Z_{1},Z_{2})\\
C(Z_{1},Z_{2}) & D(Z_{1},Z_{2})
\end{array}\right]
\]
 be a selfadjoint linearization of $P$ constructed as above, and
let $n$ be the size of its matrix coefficients. As pointed out earlier,
\[
n\tau_{n}(\ker L(X_{1},X_{2}))=\tau(\ker(P(X_{1},X_{2})).
\]
 Set $X=L(X_{1},X_{2})$ and let $q_{1}$ and $q_{2}$ be given by
Lemma \ref{lem:bigger projections}. Then $q_{1}Xq_{2}$ and $q_{2}Xq_{1}$
are again linear polynomials in $X_{1},X_{2}$ with coefficients in
$M_{n}(\mathbb C)$, but they are not selfadjoint. However, the matrix
\[
Y=\left[\begin{array}{cc}
0 & q_{1}Xq_{2}\\
q_{2}Xq_{1} & 0
\end{array}\right]
\]
 is a linear polynomial with coefficients in $M_{2n}$, it is selfadjoint,
and has kernel $\ker(q_{2}Xq_{1})\oplus\ker(q_{1}Xq_{2})$. Therefore,
Lemma \ref{lem:bigger projections}(1) implies that
\[
2n\tau_{2n}(\ker(Y))=n\tau_{n}(\ker(q_{2}Xq_{1}))+n\tau_{n}(\ker(q_{1}Xq_{2}))
\]
differs from $2\tau(\ker(P(X_{1},X_{2}))$ by an integer, and the
hypothesis implies that this is not an integer. Finally, Lemma \ref{lem:bigger projections}(2)
implies that $E_{2n}(\ker(Y))$ is invertible, and the desired conclusion
follows from Corollary \ref{cor:when E(p) is large}.
\end{proof}
Some conclusions about the variables $X_{1},X_{2}\in\widetilde{\mathcal{A}}_{{\rm sa}}$
can be drawn even in case $\tau(\ker(P(X_{1},X_{2}))=\frac{1}{2}$
and $E_{n}(p)$ is not invertible. We use the anticommutator $P(X_{1},X_{2})=X_{1}X_{2}+X_{2}X_{1}$
as an illustration. Suppose that $\lambda\in\mathbb{R}$ is an eigenvalue
of $P(X_{1},X_{2})$ such that $q=\ker(\lambda1_{\mathcal{A}}-P(X_{1},X_{2}))<1_{\mathcal{A}}$.
It follows that the operator
\[
\left[\begin{array}{ccc}
\lambda1_{\mathcal{A}} & X_{1} & X_{2}\\
X_{1} & 0 & 1_{\mathcal{A}}\\
X_{2} & 1_{\mathcal{A}} & 0
\end{array}\right]
\]
has a kernel $p$ that is the final projection of
\[
\left[\begin{array}{c}
q\\
X_{2}q\\
X_{1}q
\end{array}\right].
\]
If $E_{3}(p)$ is invertible, Corollary \ref{cor:when E(p) is large}
shows that one of the operators $X_{1},X_{2}$ has an eigenvalue.
Suppose then that neither $X_{1}$ nor $X_{2}$ has eigenvalues, so
$E_{3}(p)$ is not invertible. In this case, there exists a projection
$r\in M_{3}(\mathbb C)$ of rank one such that $(r\otimes1_{\mathcal{A}})E_{3}(p)=0$.
We have
\[
E_{3}((r\otimes1_{\mathcal{A}})p(r\otimes1_{\mathcal{A}}))=(r\otimes1_{\mathcal{A}})E_{3}(p)(r\otimes1_{\mathcal{A}})=0,
\]
and we conclude that $(r\otimes1_{\mathcal{A}})p(r\otimes1_{\mathcal{A}})=0$
and thus $(r\otimes1_{\mathcal{A}})p=0$ as well. If the vector $(\alpha,\beta,\gamma)\in\mathbb{C}^{3}$
generates the range of $r$, then
\[
\alpha q+\beta X_{1}q+\gamma X_{2}q=0.
\]
Since $X_{j}$ has no eigenvalues, we deduce that $\beta\gamma\ne0$,
so $\beta X_{1}+\gamma X_{2}$ has the eigenvalue $-\alpha$. The
argument in Proposition \ref{prop:E(p) not invertible}, applied to
the selfadjoint polynomial
\[
\left[\begin{array}{cc}
0 & \beta X_{1}+\gamma X_{2}\\
\overline{\beta}X_{1}+\overline{\gamma}X_{2} & 0
\end{array}\right]
\]
 shows now that we necessarily have $\tau(\ker(\alpha+\beta X_{1}+\gamma X_{2}))=\frac{1}{2}$.

\end{document}